\documentclass[11pt]{amsart}

\usepackage{amsfonts}
\usepackage{epsf,graphicx}
\usepackage{amscd}
\usepackage{amsmath,latexsym,amssymb,amsthm}
\usepackage{hyperref}
\usepackage{setspace,cite}
\usepackage{lscape,fancyhdr,fancybox}
\usepackage{a4}
\usepackage{dsfont,texdraw}

\def\Kmath{\mathbb{K}}

\newcounter{cnt1}
\newcounter{cnt2}
\newcounter{cnt3}
\newcommand{\blr}{\begin{list}{$($\roman{cnt1}$)$} {\usecounter{cnt1}
        \setlength{\topsep}{0pt} \setlength{\itemsep}{0pt}}}
\newcommand{\bla}{\begin{list}{$($\alph{cnt2}$)$} {\usecounter{cnt2}
        \setlength{\topsep}{0pt} \setlength{\itemsep}{0pt}}}
\newcommand{\bln}{\begin{list}{$($\arabic{cnt3}$)$} {\usecounter{cnt3}
                \setlength{\topsep}{0pt} \setlength{\itemsep}{0pt}}}
\newcommand{\el}{\end{list}}

\newtheorem{Thm}{Theorem}[section]

\newtheorem{Prop}[Thm]{Proposition}
\newtheorem{Def}[Thm]{Definition}
\newtheorem{Exm}[Thm]{Example}
\newtheorem{Rem}[Thm]{Remark}
\newtheorem{Cor}[Thm]{Corollary}
\newtheorem{Note}[Thm]{Note}

\begin{document}

\title{On exact Courant algebras}
\author{Ashis Mandal}
\begin{abstract}{
In this note we will show that exact Courant algebras over a Lie algebra $\mathfrak{g}$ can be characterised via Leibniz $2$- cocycles, and the automorphism group of a given exact Courant algebra is in a one-to-one correspondence with first Leibniz cohomology space of $\mathfrak{g}$.
}
\end{abstract}
%\date{ }
%\newsymbol \rtimes 226F
\maketitle
\section{Introduction}
The aim of this paper is to characterise exact Courant algebras over a Lie algebra $\mathfrak{ g}$ via second Leibniz cohomology space with suitable Leibniz algebra representation of $\mathfrak{ g}$, also to find a correspondence among automorphism group of a given exact Courant algebras over $\mathfrak{ g}$ and first Leibniz cohomology space of $\mathfrak{ g}$ . The notion of Courant algebra was introduced by H. Bursztyn et. al ( \cite{BCG} ) in the context of reduction of Courant algebroids and generalized complex structures. They defined an extended action of Courant algebras on Courant algebroids and explained that the moment map in symplectic geometry obtains a new interpretation as an object which controls the extended part of the action. By definition, a Courant algebra over a Lie algebra $\mathfrak{g}$  is a Leibniz algebra along with a Leibniz algebra homomorphism to $\mathfrak{g}$. Leibniz algebras were introduced by J.-L. Loday in connection with cyclic homology and Hochschild homology of matrix algebras \cite{Lo1,Lo2}. A (co)homology with coefficient in a representation associated to Leibniz algebras has been developed in \cite{LP} .  Any Lie algebra is also a Leibniz algebra as in the presence of skew symmetry of the Lie bracket the Jacobi identity for the Lie bracket is equivalent to the Leibniz identity of the bracket. Thus one can deduce the Leibniz cohomology  for a given Lie algebra with coefficients in a suitable Leibniz representation of $\mathfrak{g}$. Here we consider the underlying Lie algebra for an exact Courant algebra and its Leibniz cohomology spaces with coefficients in the Leibniz representation given by the algebra. It is important to note that an exact Courant algebra nicely fit into the consideration of  an abelian extension of the Leibniz algebra $\mathfrak{g}$ by a suitable representation.

In section $2$, we recall definitions of Leibniz algebra and its cohomology. In section $3$, we present Courant algebras over a Lie algebra. In the last section we discuss about exact Courant algebras and Leibniz cohomology in dimension one and two.
In the whole paper we assume that $\mathbb{K}$ is a fixed field,  all the vector spaces are over the field $\mathbb{K}$ and maps are $\mathbb{K}$-linear maps.

\section{Leibniz algebra and its cohomology}
In this section, we recall the definition of a Leibniz algebra and describe its cohomology.  

\begin{Def} A  Leibniz algebra (left Leibniz algebra) is a $\Kmath$-module $L$, equipped with a bracket
 operation, which is $\mathbb{K}$-bilinear and satisfies the Leibniz identity,
 $$
 [x,[y,z]]=[[x,y],z] + [y, [x,z]] \quad for \quad x,y,z \in L.
 $$
\end{Def}
In the presence of antisymmetry of the bracket operation, the Leibniz identity is
equivalent to the Jacobi identity, hence any Lie algebra is a Leibniz
algebra.
\begin{Rem}
For a Leibniz algebra defined above,  the left adjoint operations $[x,-]$ are required to be derivations for any
$x \in L.$  Analogously, Leibniz algebra ( right Leibniz algebra ) can be defined by the requirement to be satisfied is that the right adjoint map $[-,x]$ are  derivations for any $x \in L.$ In that case, the Leibniz identity in the above definition would be of the form: $$[x,[y,z]]= [[x,y],z]-[[x,z],y] ~~\mbox{for}~x,~y,~z \in L.$$ 
 Sometime, in the literature ( e.g., \cite{Lo1, Lo2, LP}) right Leibniz algebra has been considered. One need to be careful in translating from one to another. All our Leibniz algebras here will be left Leibniz algebras unless otherwise stated. 
\end{Rem}
\begin{Exm}
Let $(L,d)$ be a differential Lie algebra with the Lie bracket $[,]$. Then $L$ is a Leibniz algebra with the bracket operation $[x,y]_d:= [dx,y]$. The new bracket on $L$ is  called the derived bracket.
\end{Exm}
For more example see \cite{Lo1, Lo2, LP} , \cite {HM02, JMF09}.

Let $L$ be a Leibniz algebra and $M$ a representation of $L$. By
definition (\cite{LP,JMF09}), $M$ is a $\Kmath$-module equipped with two actions (left and
right) of $L$, 
$$
[-,-]:L \times M \rightarrow M \qquad \text{and} ~~~[-,-]:M \times L
\rightarrow M \qquad \text{such that}
$$
$$
[x,[y,z]]=[[x,y],z] + [y,[x,z]]
$$
holds, whenever one of the variables is from $M$ and the two others
from $L$.

Define $CL^n(L;M):=\text{Hom}_{\Kmath}(L^{\otimes n},M)$, $n \geq 0$.
Let
$$
\delta^n: CL^n(L;M) \rightarrow CL^{n+1}(L;M)
$$
be a $\Kmath$-homomorphism defined by

\begin{multline*}\label{Leibniz coboundary}
(\delta^n \psi)(X_1,X_2, \cdots , X_{n+1}) \\=\sum_{i=1}^{n} (-1)^{i-1} [X_i,\psi(X_1, \cdots, \hat{X_i}, \cdots, X_{n+1})]
 +(-1)^{n+1} [\psi(X_1, \cdots, X_{n}), X_{n+1}]  \\+
 \sum_{1\leqslant i < j \leqslant n+1} (-1)^{i}  
 \psi(X_1, \cdots, \hat{X_i},\cdots, X_{j-1},[X_i,X_j],X_{j+1}\cdots, X_{n+1}).
\end{multline*}
Then $(CL^*(L;M),\delta)$ is a cochain complex, whose cohomology is denoted by $HL^*(L; M)$, is
called the cohomology of the Leibniz algebra $L$ with coefficients in
the representation $M$ (see \cite{LP, JMF09}). In particular, $L$ is a representation of
itself with the obvious action given by the bracket in $L$. 

\begin{Rem}\label{Leibniz cohomology for Lie algebra}
Suppose $\mathfrak{g}$ is a Lie algebra then we may view it as a Leibniz algebra and if $\mathfrak{h}$ is a representation of $\mathfrak{g}$ equipped with two action defined above then we can talk of the Leibniz cochain complex  $(CL^{*}(\mathfrak{g}; \mathfrak{h}))$ and hence the cohomology space $HL^{*}(\mathfrak{g}; \mathfrak{h})$ .
\end{Rem}

\section{Courant algebras over a Lie algebra}
In the present section, we recall the definitions of a Courant algebra and exact Courant algebra from \cite{BCG}. Let $\mathfrak{g}$  be a Lie algebra.

\begin{Def}\label{Courant algebra}
A {\it Courant algebra }  over the Lie algebra $\mathfrak{g}$ is a vector space $\mathfrak{a}$ equipped with a bilinear bracket $[-,-]:\mathfrak{a} \times \mathfrak{a} \longrightarrow \mathfrak{a}$ and a map $\pi: \mathfrak{a} \longrightarrow \mathfrak{g}$, which satisfy the following conditions for all $a_1,a_2,a_3 \in \mathfrak{a}:$

\begin{equation}\label{Courant alg condition}
 \begin{split}
  &[a_1,[a_2, a_3]]=[[a_1,a_2],a_3]+ [a_2,[a_1,a_3]],\\
  & \pi([a_1,a_2])= [\pi(a_1), \pi(a_2)].
 \end{split}
\end{equation}
In other words, $\mathfrak{a}$ is a Leibniz algebra with a homomorphism $\pi$  from $\mathfrak{a}$ to $\mathfrak{g}$. We will simply denote it as $\pi: \mathfrak{a} \longrightarrow \mathfrak{g}$ is a Courant algebra.
\end{Def}
A Courant algebroid over a smooth manifold $M$  gives an example of a Courant algebra over the Lie algebra of vector fields $\mathfrak{g}=\Gamma(TM)$ by taking $\mathfrak{a}$ as the Leibniz algebra structure on the space of sections of the underlying vector bundle of the Courant algebroid.  From \cite{RW} it follows that any Courant algebra is actually an example of a $2$- term $L_\infty$ algebra \cite{BC,St}.

\begin{Def}\label{exact Courant algebra}
An  {\it exact} Courant algebra  over the Lie algebra $\mathfrak{g}$ is a Courant algebra $\pi: \mathfrak{a} \longrightarrow \mathfrak{g}$ for which $\pi$ is a surjective linear map and $\mathfrak{h}=ker (\pi)$ is abelian, i.e. $[h_1,h_2]=0$ for all $h_1,h_2 \in \mathfrak{h}$.
\end{Def}
For an exact Courant algebra $\pi: \mathfrak{a} \longrightarrow \mathfrak{g}$, there are actions of  $\mathfrak{g}$ on $\mathfrak{h}$: 
$[g,h]=[a,h]$ and $[h,g]=[h,a]$ for any $a$ such that $\pi(a)=g$. This defines $\mathfrak{h}$ as a representation of  $\mathfrak{g}$ viewed as a Leibniz algebra.

The next example will give a natural exact Courant algebra associated with any representation of the Lie algebra $\mathfrak{g}$.

\begin{Exm} (Hemisemidirect product). Let $\mathfrak{g}$ be a Lie algebra acting on the vector space $\mathfrak{h}$. Then 
$\mathfrak{a}=\mathfrak{g}\oplus \mathfrak{h}$ becomes an exact Courant algebra over $\mathfrak{g}$ via the bracket defined for $(g_1,h_1),(g_2,h_2) \in \mathfrak{a}$ as follows.
$$ [(g_1,h_1),(g_2,h_2)]=([g_1,g_2], g_1. h_2),$$
where $g.h$ denotes the action of the Lie algebra $\mathfrak{g}$ on $\mathfrak{h}$.
\end{Exm}
This bracket is called {\it hemisemidirect} product in \cite{KW}. If we consider above a Leibniz representation $\mathfrak{h}$ of the Lie algebra $\mathfrak{g}$, then 
$\mathfrak{a}=\mathfrak{g}\oplus \mathfrak{h}$ becomes an exact Courant algebra over $\mathfrak{g}$ via the bracket 
$$ [(g_1,h_1),(g_2,h_2)]=([g_1,g_2], [g_1, h_2]+[h_1,g_2]),$$
where the actions (left and right) of $\mathfrak{g}$ on $\mathfrak{h}$ is denoted by the same bracket $[-,-]$.
Next result will describe more exact Courant algebras on $\mathfrak{a}=\mathfrak{g}\oplus \mathfrak{h}$.
\begin{Prop}\label{algebra induced by cocycle}
Every $2$-cocycle in $CL^2(\mathfrak{g}; \mathfrak{h})$ defines an exact Courant algebra structure over $\mathfrak{g}$.
\end{Prop}
\begin{proof}
Let $f \in CL ^2 (\mathfrak{g};\mathfrak{h})$ be a $2$-cocycle. Consider $\mathfrak{a}=\mathfrak{g} \oplus \mathfrak{h}$ and define a bracket operation 
$$[(g_1, h_1) ,(g_2, h_2)]=([g_1, g_2],~ [g_1, h_2] + [g_2 , h_1]+f(g_1, g_2))$$  for $(g_1,h_1), (g_2,h_2) \in \mathfrak{a}$. 

Since $\delta^2 f =0$, and $\mathfrak{h}$ is representation of $\mathfrak{g}$, the bracket defined above satisfies Leibniz identity on $\mathfrak{a}$. Thus $ \mathfrak{a} \stackrel{\pi}{\longrightarrow}\mathfrak{g}$ is an exact Courant algebra over the Lie algebra  $\mathfrak{g}$, where  $\pi(g,h)=g$ for $(g,h) \in \mathfrak{a}$.
\end{proof}
\begin{Note}
Suppose $\pi: \mathfrak{a} \longrightarrow \mathfrak{g}$ is an exact Courant algebra.  If we fix a section of the $\mathbb{K}$-linear map $\pi$, then we get an abelian extension  of the Leibniz algebra $\mathfrak{g}$ by its representation $\mathfrak{h}$:
$$0\longrightarrow {\mathfrak{h}}\stackrel{i}{\longrightarrow} \mathfrak{a} \stackrel{\pi}{\longrightarrow}\mathfrak{g}\longrightarrow 0,$$
where $i$ denotes the inclusion map (see  \cite{LP}). Later on we will use this abelian extension form for an exact Courant algebra.
\end{Note}
Let $\pi: \mathfrak{a} \longrightarrow \mathfrak{g} $ and $\pi^\prime: \mathfrak{a}^\prime \longrightarrow \mathfrak{g}$ be two such  exact Courant algebras with $ker (\pi)=\mathfrak{h}=ker (\pi^\prime)$. 

A $\mathbb{K}$-linear map $F: \mathfrak{a} \longrightarrow \mathfrak{a}^\prime$  is said to be an exact Courant algebra homomorphism if it is a Leibniz algebra homomorphism  such that following diagram commutes.
\[
 \begin{CD}
 0 @>>> \mathfrak{h} @>{i}>> \mathfrak{a} @>{\pi}>> \mathfrak{g} @>>> 0\\
          @.    @|  @V{F}VV @|    \\
0 @>>> \mathfrak{h} @>{i}>>\mathfrak{a}^\prime @>{\pi^\prime}>> \mathfrak{g} @>>> 0
 \end{CD}
\]
Likewise, we can define an isomorphism of exact Courant algebras. We present a description of  the set of all isomorphism classes of  a given exact Courant algebra over a Lie algebra.
\section{Exact Courant algebras and Leibniz cohomology in dimension one and two}
Suppose $\mathfrak{g}$ is a Lie algebra which is now viewed as a Leibniz algebra and $\mathfrak{h}$ is an abelian Lie algebra endowed with two actions (left and right ) denoted by the same bracket notation $[-,-]$ give rise to a representation of $\mathfrak{g}$.

\begin{Prop}\label{induced by a class}
Every cocycle in the same cohomology class in $HL ^2 (\mathfrak{g};\mathfrak{h})$  defines an isomorphic exact Courant algebra over $\mathfrak{g}$.
\end{Prop}
\begin{proof}
Let $f$ be a representative of the class  $[f]\in  HL ^2 (\mathfrak{g};\mathfrak{h})$.  Then by Proposition \ref{algebra induced by cocycle}, we get an exact Courant algebra:
$$0\longrightarrow {\mathfrak{h}}\stackrel{i}{\longrightarrow} \mathfrak{a} \stackrel{\pi}{\longrightarrow}\mathfrak{g}\longrightarrow 0.$$
Suppose we take another representative $f^\prime$ of the class $[f]\in HL ^2 (\mathfrak{g};\mathfrak{h})$ and get another exact Courant algebra given by,
$$0\longrightarrow {\mathfrak{h}}\stackrel{i}{\longrightarrow} \mathfrak{a^\prime} \stackrel{\pi}{\longrightarrow}\mathfrak{g}\longrightarrow 0.$$
 Now $f -f^\prime=\delta^1 \psi$ for some $\psi \in Hom(\mathfrak{g}; \mathfrak{h})$. Let us define a $\mathbb{K}$- linear map $F:\mathfrak{a}\longrightarrow \mathfrak{a}^\prime$  by $F(g, h)=(g, h +\psi(g))$. Then it follows that $F$ is a Leibniz algebra homomorphism such that following diagram commutes.
 \[
 \begin{CD}
 0 @>>> \mathfrak{h} @>{i}>> \mathfrak{a} @>{\pi}>> \mathfrak{g} @>>> 0\\
          @.    @|  @V{F}VV @|    \\
0 @>>> \mathfrak{h} @>{i}>>\mathfrak{a}^\prime @>{\pi^\prime}>> \mathfrak{g} @>>> 0
 \end{CD}
\]
This gives an isomorphism of the exact Courant algebras defined by the cocycles $f$ and $f^\prime$. Thus $[f]\in HL ^2 (\mathfrak{g};\mathfrak{h})$ corresponds to an isomorphism classes of exact Courant algebra over $\mathfrak{g}$.
\end{proof}

 The next result will characterise an exact Courant algebra over $\mathfrak{g}$ via its second Leibniz cohomology  space. 
\begin{Thm}\label{2 class and exact Courant algebras}
Elements of $ HL ^2 (\mathfrak{g};\mathfrak{h})$  corresponds bijectively to isomorphism classes of exact Courant algebras
$$0\longrightarrow {\mathfrak{h}}\stackrel{i}{\longrightarrow} \mathfrak{a} \stackrel{\pi}{\longrightarrow}\mathfrak{g}\longrightarrow 0.$$
\end{Thm}
\begin{proof}
By Proposition \ref{induced by a class}, we find an isomorphism class of exact Courant algebras associated to any element of $ HL ^2 (\mathfrak{g};\mathfrak{h})$.

 We now consider an exact Courant algebra
$$0\longrightarrow {\mathfrak{h}}\stackrel{i}{\longrightarrow} \mathfrak{a} \stackrel{\pi}{\longrightarrow}\mathfrak{g}\longrightarrow 0.$$
 %$\pi: \mathfrak{a} \longrightarrow \mathfrak{g}$.

Fix a section $q:\mathfrak{g}\longrightarrow \mathfrak{a}$ of the projection $\pi$. \\ Define $$\alpha:\mathfrak{a}\longrightarrow \mathfrak{g}\oplus \mathfrak{h}~~ by ~~\alpha(x)=(\pi(x), x -q\circ \pi(x)).$$
 From definition $\alpha$ is a $\mathbb{K}$-linear map.  Now if $\alpha(x)=(0,0)$ then $\pi(x)=0$ and $x-q\circ \pi(x)=0$, hence $ x=0$.  So $\alpha$ is an injective linear map.
Suppose $(g,h)\in \mathfrak{g}\oplus \mathfrak{h}$, take $x=q(g)+h$, then $\pi(x)=g$ and $x-q\circ \pi(x)= q(g)+h-q(g)=h$.  Therefore $\alpha(x)=\alpha(q(g)+h)=(g,h)$, showing $\alpha$ is an onto map. Thus $\alpha$ is a $\mathbb{K}$-module isomorphism.
Let $(g,h)_q$ denotes the inverse image $q(g)+h$ of $(g,h) \in \mathfrak{g} \oplus \mathfrak{h} $ under the above isomorphism $\alpha$. 
Define $\phi_q:\mathfrak{g}^{\otimes 2}\longrightarrow \mathfrak{h}$ by 
\begin{equation}\label{defn of phiq}
\phi_q(g_1,g_2)= [(g_1,0)_q,(g_2,0)_q]-([g_1, g_2],0)_q = [q(g_1), q(g_2)] - q([g_1,g_2])
\end{equation}
for $g_1,g_2\in \mathfrak{g}$.
Now for $(g_1,h_1)_q$ and $(g_2,h_2)_q$ in $\mathfrak{a}$ we can write down the bracket as follows.
\begin{equation*}
 \begin{split}
&~[(g_1,h_1)_q, (g_2,h_2)_q]\\
%=&~[q(g_1)+ h_1, q(g_2)+h_2]\\
=&~[q(g_1),q(g_2)]+[q(g_1), h_2]+[h_1, q(g_2)]+[h_1,h_2]\\
=&~[q(g_1),q(g_2)]+[\pi \circ q(g_1), h_2]+[h_1, \pi \circ q(g_2)] \\
%=&~q([g_1,g_2])+\{[q(g_1),q(g_2)]-q([g_1,g_2])\}+  [g_1, h_2]+[h_1,g_2]\\
=&~q([g_1,g_2])+ [g_1, h_2]+[h_1,g_2]+\phi_q(g_1,g_2)\\
=&~([g_1,g_2]~,~[g_1, h_2]+[h_1,g_2]+\phi_q(g_1,g_2))_q.
\end{split}
\end{equation*}
Thus the bracket operation in $\mathfrak{a} $ can be expressed as,
$$[(g_1,h_1)_q, (g_2,h_2)_q] = ~([g_1,g_2]~,~[g_1, h_2]+[h_1,g_2]+\phi_q(g_1,g_2))_q$$
Suppose $(g_1,h_1)_q,(g_2,h_2)_q$ and $(g_3,h_3)_q \in \mathfrak{a}$, using the Leibniz identity of the algebra multiplication in $\mathfrak{a}$, we get,
\begin{equation*}
 \begin{split}
&[( g_1,h_1)_q , [ (g_2,h_2)_q, (g_3,h_3)_q ]] \\
=&[ [ (g_1,h_1)_q, (g_2,h_2)_q], (g_3,h_3)_q ]+ [ (g_2,h_2)_q, [(g_1,h_1)_q ,(g_3,h_3)_q] ]=0.
\end{split}
\end{equation*}
or,  $\delta^2 \phi_q(g_1,~g_2,~g_3)= 0$. So $\phi_q \in  CL^2(\mathfrak{g};\mathfrak{h})$ is a cocycle. 

Let $q':\mathfrak{g}\longrightarrow \mathfrak{a}$ be another section of $\pi$. Replacing $q$ by $q^\prime$ in the above argument, will give rise to a cocycle $\phi_{q^\prime}\in CL^2(\mathfrak{g};\mathfrak{h})$.
Set $\beta= (q'-q) \in  Hom(\mathfrak{g}; \mathfrak{h})$. Then $\phi_{q'}-\phi_{q}=\delta^1 \beta$. Thus for a given exact courant algebra $\pi: \mathfrak{a} \longrightarrow \mathfrak{g}$  there is a unique cohomology class $[\phi_q]$ in $HL^2(\mathfrak{g};\mathfrak{h})$.

Let 
$ \pi^\prime: \mathfrak{a}^\prime \longrightarrow \mathfrak{g}$ be another exact Courant algebra over $\mathfrak{g}$ with the same linear space $\mathfrak{h}$ as kernel of $\pi^\prime$. assume that it is isomorphic to the exact Courant algebra $\mathfrak{a}$ we took before. Let $F : \mathfrak{a}\longrightarrow \mathfrak{a}^\prime$ be an isomorphism of exact Courant algebras  $\mathfrak{a}$ and $\mathfrak{a}^\prime$. Thus the following diagram commutes.
\[
 \begin{CD}
 0 @>>> \mathfrak{h} @>{i}>> \mathfrak{a} @>{\pi}>> \mathfrak{g} @>>> 0\\
          @.    @|  @V{F}VV @|    \\
0 @>>> \mathfrak{h} @>{i}>>\mathfrak{a}^\prime @>{\pi^\prime}>> \mathfrak{g} @>>> 0
 \end{CD}
\]
 Then $q^\prime =f \circ q :\mathfrak{g} \longrightarrow \mathfrak{a}^\prime$ is a section for $\pi^\prime$. We have $\phi_{q^\prime}:\mathfrak{g}^{\otimes 2} A\longrightarrow \mathfrak{h}$ defined by 
$$\phi_{q^\prime}(g_1,g_2)=[q^\prime(g_1), q^{\prime}(g_2)] - q^\prime([g_1, g_2]) ~(\mbox{cf.}~ (\ref{defn of phiq}))$$
%Thus $\phi_{q^\prime}(a_1,a_2)=(i^\prime)^{-1}(q^\prime(a_1)q^{\prime}(a_2)-q^\prime(a_1a_2))=i^{-1}(q(a_1)q(a_2)-q(a_1a_2))=\phi_q(a_1,a_2).$
 Consequently $\phi_{q}$ and $\phi_{q^\prime}$ represents the same class in $HL^2(\mathfrak{g}; \mathfrak{h})$.
\end{proof}
\begin{Cor}
The characteristic element of a Leibniz algebra $\mathfrak{a}$ defines an isomorphism class of an exact Courant algebra over the Lie algebra obtained by taking quotient with the Leibniz kernel of $\mathfrak{a}$. 
\end{Cor}

Let $\mathcal{A}$ denote the group of all automorphisms of any given exact Courant algebra over a Lie algebra $\mathfrak{g}$ . 
\begin{Prop}\label{automorphism group is equivalent to}
There is a one-to-one correspondence between the group $\mathcal{A}$ of automorphisms of any given exact Courant algebra ,
$ \pi: \mathfrak{a} \longrightarrow \mathfrak{g}$  with $ker \pi= \mathfrak{h}$  and $HL^1(\mathfrak{g};\mathfrak{h})$.
\end{Prop}
\begin{proof}
 Suppose $F:\mathfrak{a} \longrightarrow \mathfrak{a}$ is a $\mathbb{K}$-algebra isomorphism giving an automorphism of the given exact Courant algebra,
 $ \pi: \mathfrak{a} \longrightarrow \mathfrak{g}$  with $ker \pi= \mathfrak{h}$. 
Thus we get the following commutative diagram.
\[
 \begin{CD}
 0 @>>> \mathfrak{h} @>{i}>> \mathfrak{a} @>{\pi}>> \mathfrak{g} @>>> 0\\
          @.    @|  @V{F}VV @|    \\
0 @>>> \mathfrak{h} @>{i}>>\mathfrak{a} @>{\pi}>> \mathfrak{g} @>>> 0
 \end{CD}
\]
Now fix a section $q:\mathfrak{g} \longrightarrow \mathfrak{a}$ of $\pi$.
As in the proof of Proposition \ref{2 class and exact Courant algebras}, we have a linear isomorphism $\alpha: \mathfrak{a}\cong \mathfrak{g}\oplus \mathfrak{h}$ and let $(g,h)_q$ denote the inverse image $q(g)+h$ of $(g,h)$ under this isomorphism.

Suppose  $F((g,h)_q)=(F_1((g,h)_q),F_2((g,h)_q))_q$  for $(g,h)_q \in \mathfrak{a}$, where $F_1, F_2$ are maps obtained from $\alpha \circ F$ by taking projection into first and second components. 
By the above diagram we have $\pi \circ F=\pi$ and $F(h)=h$ for $h \in \mathfrak{h}$. These in turn give $F_1((g,h)_q)=g$ and   $F_2((0,h)_q)=h$ respectively. Therefore for $(g,h)_q \in \mathfrak{a}$, we get  
\begin{equation*}
 \begin{split}
&F((g,h)_q)=(F_1((g,h)_q),F_2((g,h)_q))_q\\
\mbox{or,}~&F((g,0)_q+(0,h)_q)=(g,F_2((g,h)_q))_q\\
\mbox{or,}~&F((g,0)_q)+F((0,h)_q)=(g,F_2((g,h)_q))_q\\
\mbox{or,}~&(g,F_2(g,0)_q)_q+(0,h)_q=(g,F_2((g,h)_q))_q\\
\mbox{or,}~&(g,F_2((g,0)_q)+h)_q=(g,F_2((g,h)_q))_q.
 \end{split}
\end{equation*}
So
$ F_2((g,h)_q)=\psi(g)+h $ 
 where the map $\psi:\mathfrak{g} \longrightarrow \mathfrak{h}$ is given by $\psi(g)=F_2((g,0)_q)$.
Let $g_1,g_2 \in \mathfrak{g}$. Then $(g_1,0)_q, (g_2,0)_q\in \mathfrak{a}$.
 Since  $F$ is a Leibniz algebra homomorphism, we have,
 \begin{equation}
 \begin{split} 
&  F( [(g_1,~0)_q, (g_2,~0)_q] ) = [ F((g_1,~0)_q), F((g_2,~0)_q)] \\
\mbox{or,}~ &F(([ g_1, g_2],~ \phi_q(g_1,~g_2))_q) 
=[ (g_1,~F_2((g_1,~0)_q))_q , (g_2,~F_2((g_2,~0)_q))_q] \\
\mbox{or,}~ &([g_1, g_2],~\phi_q(g_1,~g_2)+\psi(g_1,~g_2))_q = (g_1,~\psi(g_1))_q(g_2,~\psi(g_2))_q \\
 \mbox{or,}~ &([g_1, g_2],~\phi_q(g_1,~g_2)+\psi(g_1,~g_2))_q = ([g_1,g_2],~[g_1, \psi(g_2)]+[\psi(g_1), g_2]+\phi_q(g_1,g_2))_q \\
\mbox{or,}~ &\psi([g_1,g_2])=[g_1, \psi(g_2)]+ [\psi(g_1), g_2]\\
\mbox{or,}~ &\delta^1 \psi(g_1,g_2)=0.
 \end{split}
 \end{equation}
 Therefore $\psi$ represents a cohomology class in $ HL^1(\mathfrak{g};\mathfrak{h})$.\\
 Conversely,  suppose $\psi:\mathfrak{g} \longrightarrow \mathfrak{h}$ is a linear map with $\delta^1 \psi=0$. 
Define, $F: \mathfrak{a} \longrightarrow \mathfrak{a}$ by $F((g,h)_q)=(g,h+\psi(g))_q$. 
Then $F$ is a Leibniz algebra isomorphism and it gives an automorphism of the given exact Courant algebra. 

Consequently, the assignment $ F \mapsto [\psi]$ is the required bijection. 
 \end{proof}

%\newpage

\vspace{.5cm}
{\bf Ashis Mandal}\\
 Department of Mathematics and Statistics,\\
Indian Institute of Technology,\\
Kanpur 208016, India.\\
e-mail: amandal@iitk.ac.in

\end{document}